\documentclass[reqno]{amsart}

\usepackage[latin1]{inputenc}

\usepackage{amssymb}
\usepackage{amsmath}
\usepackage{amscd}
\usepackage{amsthm}
\usepackage{amsfonts}
\usepackage{graphicx}
\usepackage{url}
\usepackage[breaklinks=true,hyperref]{hyperref}
\usepackage{amssymb}
\usepackage[dvips]{color}
\usepackage{epsfig}
\usepackage{mathrsfs}

\newtheorem{theorem}{Theorem}[section]

\newtheorem{proposition}[theorem]{Proposition}

\newtheorem{corollary}[theorem]{Corollary}

\theoremstyle{definition}

\newtheorem{definition}[theorem]{Definition}

\theoremstyle{remark}

\numberwithin{equation}{section}

\newcommand{\Pb}{\mathcal{P}}
\newcommand{\He}{\mathcal{H}}
\newcommand{\C}{\mathcal{\mathbf{C}}}
\newcommand{\R}{\mathcal{\mathbf{R}}}

\begin{document} 

\title{Strongly Cancellative and Recovering Sets On Lattices}
\author{ShinnYih Huang}
\author{Hoda Bidkhori}

\begin{abstract}
We use information theory to study recovering sets $\R_L$ and strongly cancellative sets $\C_L$ on different lattices. These sets are special classes of recovering pairs and cancellative sets previously discussed in [1], [3] and [5]. We mainly focus on the lattices $B_n$ and $D_{l}^{k}$. Specifically, we find upper bounds and constructions for the sets $\R_{B_n}$, $\C_{B_n}$, and $\C_{D_{l}^{k}}$.
\end{abstract}

\maketitle

\section{\textbf{Introduction}}
In this paper, we study the strongly cancellative sets $\C_{L}$ and recovering sets $\R_L$ which are subsets of points in lattices $L$, see Definition \ref{drecs} and \ref{dcans}. The study of strongly cancellative sets is motivated by the work of Frankl and Furedi [3] and Fredman [6] on cancellative sets. Specifically, strongly cancellative sets are a special class of cancellative sets. 

On the other hand, the study of the recovering sets is prompted by the previous work of Simonyi [1] on recovering pairs. A recovering pair $(A,B)$ is an ordered pair of subsets $A,B$ of points in a lattice such that for any $a,a' \in A$ and $b,b' \in B$, we have the follwing:
\begin{align*}
a \wedge b &= a' \wedge b'  \Rightarrow a = a', \\
a \vee b &= a' \vee b' \Rightarrow b = b'.
\end{align*}
The paper [7] of K\"orner and Olistky shows that the upper bound of $|A||B|$ plays an important role in the zero-error information theory. Simonyi gave an upper bound $3^n$ for the size of $|A||B|$ on the Boolean lattice, and Holzman and K\"orner improved the bound to $2.3264^n$ afterward. Through out this paper, we study a special class of the recovering pairs $(\R_{L},\R_{L})$ which has the same set $\R_{L}$. In this case, we call $\R_{L}$ a recovering set.

As we go through this paper, one can see in Definition \ref{drecs} and \ref{dcans} that recovering sets are also a special case of strongly cancellative sets. We focus on the upper bounds and structures of these two sets by using some results in Infomation Theory introduced by Holzman and K\"orner [4].

This paper is organized as follows: In Section 2, we go through the definitions of strongly cancellative sets and recovering sets and some results on the entropy function in information theory. In Section 3, we study the recovering set $\R_{B_n}$ on the Boolean lattice $B_n$ and give an upper bound $\left|\R_{B_n}\right| \leq \sqrt{3} \cdotp 2^{0.4392n}$. As a result, this class of the recovering pairs has an upper bound $3 \cdotp 2^{0.8784n} = 3 \cdotp (1.7992554)^n$ on its size. In Section 4, we study strongly cancellative sets $\C_{B_n}$ on $B_n$. We give a tight upper bound $2^{\lfloor \frac{n}{2} \rfloor}$ on $|\C_{B_n}|$ for this lattice. Finally, in section 5, we consider the strongly cancellative sets $\C_{D_{l_1,\ldots,l_k}}$ on the lattice $D_{l_1,\ldots,l_k}$ which is the product of $k$ chains of lengh $l_1-1 , \ldots, l_k-1$. We show that when $l_1=\cdots=l_k = l$, there exists a strongly cancellative set of size $l^{\left\lfloor \frac{k}{2} \right\rfloor }$ and $|\C_{D_{l,\ldots,l}}| \leq (2 l)^{\frac{k}{2}} + \frac{k(l-1)}{2} +1$.

\section{\textbf{Preliminaries}}
For basic definitions and results concerning lattices, we encourage readers to consult Chapter 3 of [10]. In particular, the Boolean lattice $B_n$ is the lattice  of all subsets of the set $\{1,\ldots,n \}$ ordered by inclusion, and $D_{l_1,\ldots,l_k}$ is the lattice formed by thr product of $k$ chains of lengh $l_1-1 , \ldots, l_k-1$, so that the points in $D_{l_1,\ldots,l_k}$ correspond to $k$-dimensional vectors $(v_1,\ldots,v_k)$ with $0\leq v_i\leq l_i-1$. The ordering of points in $D_{l_1,\ldots,l_k}$ is as follows:
$$v \preceq w \Leftrightarrow v_i \leq w_i\text{, for all }1\leq i\leq k.$$

A cancellative set is a subset of points in lattice $L$ that any three different points $v_1,v_2,v_3$ in this set satisfy the following condition:
$$v_1 \wedge v_2 \neq v_1 \wedge v_3.$$
We define strongly cancellative sets as a special class of cancellative sets.

\begin{definition}\label{dcans}
A \textit{strongly cancellative set} $\C_{L}$ of lattice $L$ is a subset of points in $L$ such that for any three different points $a_1,a_2,a_3 \in \C_{L}$, 
\begin{equation}\label{e6}
a_1 \wedge a_2 \neq a_1 \wedge  a_3 \text{ and } a_1 \vee a_2 \neq a_1 \vee  a_3.
\end{equation}
\end{definition}

A recovering set holds all all the conditions which define a strongly cancellative set. In addition, any recovering set $\R_{L}$ forms a recovering pair $( \R_{L}, \R_{L} )$ on $L$. As a result, the points in $\R_{L}$ satisfy the following conditions:

\begin{definition}\label{drecs}
A \textit{recovering set} $\R_{L}$ of lattice $L$ is a subset of points in $L$ such that for any four different points $a_1,a_2,a_3,a_4 \in \R_{L}$, we have
\begin{align}
a_1 \wedge a_2 \neq a_3 \wedge  a_4 \text{ and } a_1 \vee a_2 \neq a_3 \vee  a_4, \\
a_1 \wedge a_2 \neq a_1 \wedge  a_3 \text{ and } a_1 \vee a_2 \neq a_1 \vee  a_3.
\end{align}
\end{definition}

Now, we  introduce the \textit{entropy function} and show an inequality of it.

Given a discrete random variable $X$ with $m$ possible values $x_1,\ldots,x_m$, we define the \textit{entropy function} $\He$ of $X$ as follows:
\begin{equation}\label{eentropy}
\He(X) = - \sum_{i=1}^{m} p(x_i)\log_{b}{p(x_i)} = \sum_{i=1}^{m} p(x_i)\log_{b}\frac{1}{p(x_i)},
\end{equation}
where $p$ is the probability mass function of $X$ and $x_i$ is the value of $X$. 
In this paper, we always set $b = 2$. In this case, the function $x\log{\frac{1}{x}}$ is concave down when $x > 0 $. Therefore, for any $s$ values $0\leq p_1,\ldots,p_s \leq 1$, we have
\begin{equation}\label{ecommonb}
\sum_{j=1}^{s} \left(p_j\log{\frac{1}{p_j}} \right) \leq s\cdot \left(\frac{\sum_{j=1}^{s} p_j}{s}\right) \cdotp \log{\left(\frac{s}{\sum_{j=1}^{s}p_j} \right)}.
\end{equation} 

The following inequality of entropy functions is the mahor inequality throughout this paper. A proof of the inequality is given in [9].

\begin{theorem}\label{eentropyin}
If $\xi = (\xi_1,\ldots,\xi_m)$ is an $n$-dimensional random varialbe, then
\begin{equation}
\He(\xi) \leq \sum_{i=1}^{n} \He(\xi_i).
\end{equation}
\end{theorem}

\section{\textbf{Recovering Set on $B_n$}}
In this section, we study recovering sets on Boolean lattice $B_n$. Since we are considering the Boolean lattice, we use $\cap$ and $\cup$ instead of $\wedge$ and $\vee$. In the following theorem, we give an upper bound for  $\left|\R_{B_n}\right|$.

\begin{theorem}\label{tinrs}
For any recovering set $\R_{B_n}$, we have $|\R_{B_n}| \leq \sqrt{3}\cdotp 2^{0.4392n}$.
\end{theorem}

\begin{proof}
We define random variable $\xi = a_i \cap a_j$, where $a_i$ and $a_j$ are independently chosen according to the uniform distribution on $\R_{B_n}$. Clearly, there are $|\R_{B_n}|^2$ possible pairs $(a_i,a_j)$. We wish to show that for any value $a$ in $\xi$, there are at most three pairs $(a_i,a_j)$ such that $a= a_i \cap a_j$. Given a pair $(a_t, a_s)$, we have the following two cases:

\begin{enumerate}
\item $a_t \neq a_s$. Suppose that there exists another pair $(a_{t_1},a_{s_1})$, where $a_{t_1}\cap a_{s_1} = a_t\cap a_s = a_s\cap a_t$. By Definition \ref{drecs}, $a_{t_1}$ and $a_{s_1}$ should be the same element in $B_n$, and we have the following possible cases:

\begin{enumerate}
\item $a_{t_1} = a_{s_1} \notin \{a_t,a_s\}$. 

In this case, since $a_t \cap a_s = a_{t_1} \cap a_{t_1}$, $a_{t_1}$ is contained in $a_t$ and $a_s$. Therefore, $a_{t_1} \cap a_t = a_{t_1} \cap a_s$ which contradicts the second requirement of Definition \ref{drecs}. Therefore, this case is not possible.
\item $a_{t_1} = a_{s_1} \in \{a_t,a_s\}$.

Since $a_{t_1} = a_t\cap a_s$ and $a_t \neq a_s$, it is easy to see that $(a_t,a_s)$ is either equal to $(a_i,a_i)$ or $(a_j,a_j)$.
\end{enumerate}
(a) and (b) imply that when $a_t \neq a_s$, there is at most one additional pair $(a_{t_1},a_{s_1})$ that $a_{t_1} \cap a_{s_1} = a_t \cap a_s = a_s \cap a_t$, and so at most these three pairs.

\item $a_t = a_s$. One can easily see that this is the same condision as case (b) in (1). That is to say, $a_t$ is either $a_{t_1}$ or $a_{s_1}$, and $(a_{t_1},a_{s_1})$ has only two possible choices.
\end{enumerate}

Consequently, there are at most three pairs have the same intersection value in $\xi$. Using this property, we give a lower bound on the entropy function of $\xi$.

For any $a$ in $\xi$, let $\mathbf{C}(a) = \{(a_i,a_j) : a_i\cap a_j = a\text{, and }a_i,a_j \in \R_{B_n} \}$. The probability that $\xi = a$ is $\Pb_a = \frac{|\mathbf{C}(a)|}{|\R_{B_n}|^2}$. By the above arguemnt, $\Pb_a \leq \frac{3}{|\R_{B_n}|^2}$. Considering the entropy function defined in \eqref{eentropy}, we obtain the following inequality:
$$
\mathcal{H}(\xi) = \sum_{a\in \xi}{\Pb_a\log{\frac{1}{\Pb_a}}} \geq \sum_{a\in \xi}{\Pb_a\log{\frac{|\R_{B_n}|^2}{3}}} = \log{\frac{|\R_{B_n}|^2}{3}}.
$$

On the other hand, $\xi$ is an $n$-dimensional random variable ($\xi_1,\ldots,\xi_n$), where 
$$
\xi_t = 
\begin{cases} 
  1,  & t \in a_i\cap a_j. \\
  0, & t \notin a_i\cap a_j. 
\end{cases}
$$

We set $\R_{B_n}(t) = \{a_i \mid a_i \in \R_{B_n}\text{, }t\in a_i\}$ and $\Pb_{\R_{B_n}}(t) = \frac{|\R_{B_n}(t)|}{|\R_{B_n}|}$, for any $1 \leq t \leq n$. Therefore, for any $t \in \{1,\ldots,n \}$, the probability that $\xi_t = 1$ is $\left(\Pb_{\R_{B_n}}(t)\right)^2$. Let us denote the function $h(x)$ as $x\log{\frac{1}{x}} + (1-x)\log{\frac{1}{1-x}}$. We here by Theorem \ref{eentropyin} that
\begin{equation}\label{ierecs1}
\log{\frac{|\R_{B_n}|^2}{3}} \leq \mathcal{H}(\xi) \leq \sum_{t = 1}^{n} \mathcal{H}(\xi_{t}) = \sum_{t = 1}^{n}\left[h\left(\Pb_{\R_{B_n}}(t)^2\right)\right],
\end{equation}

By considering the random variable $\xi^{'} = a_i \cup a_j$, we similarly get 
\begin{equation}\label{ierecs2}
\log{\frac{|\R_{B_n}|^2}{3}} \leq  \sum_{t = 1}^{n}h\left(\left(1-\left(\Pb_{\R_{B_n}}(t)\right)\right)^2\right).
\end{equation}

We  average over \eqref{ierecs1} and \eqref{ierecs2} to obtain a better upper bound for $\log{\frac{|\R_{B_n}|^2}{3}}$, namely:
\begin{align}
\log{\frac{|\R_{B_n}|^2}{3}} &\leq \frac{1}{2} \sum_{t = 1}^{n}\left[h\left(\Pb_{\R_{B_n}}(t)^2\right) + h\left(\left(1-\Pb_{\R_{B_n}}(t)\right)^2\right)\right] \\
&\leq \frac{n}{2}\left[\max_{0\leq x \leq 1}{\left(h(x^2) + h\left((1-x)^2\right)\right)}\right].
\end{align}

Furthermore, since $h'(x) = \log{\frac{1-x}{x}}$ is positive when $x \leq \frac{1}{2}$, and negative when $x > \frac{1}{2}$. Hence, $h(x)$ increases in the interval $\left[0,\frac{1}{2}  \right]$ and decreases in the interval $\left[\frac{1}{2},1  \right]$. We have the following two cases:

\begin{enumerate}

	\item $x \geq \frac{6}{11}$ or $x \leq \frac{5}{11}$. Then $h(x^2) + h((1-x)^2) \leq 1 + h(\frac{25}{121}) \leq 1.7349558$.

	\item $\frac{5}{11} \leq x \leq \frac{6}{11}$. Then $h(x^2) + h((1-x)^2) \leq 2h(\frac{36}{121}) \leq 1.7564781$.
\end{enumerate}

By combining (1) and (2), we have 
$$\log{\frac{|\R_{B_n}|}{\sqrt{3}}} \leq \frac{n}{4}\left[\max_{0\leq x \leq 1}{\left(h(x^2) + h\left((1-x)^2\right)\right)}\right] \leq 0.4392n.$$

Therefore, 
$$|\R_{B_n}| \leq \sqrt{3} \cdotp 2^{0.4392n}.$$

\end{proof}

One can see that Theorem \ref{tinrs} gives an upper bound $\left(\sqrt{3} \cdotp 2^{0.4392n} \right)^2 = 3 \cdotp 2 ^ {0.8784n}$ for $|\R_{B_n}|^2$ concerning the special class $(\R_{B_n},\R_{B_n})$ of recovering pairs on the Boolean lattice. This result shows a significant improvement of the cardinality of general recovering pairs discussed in [1], [4], and [5].

\section{\textbf{Strongly Cancellative set on $B_n$}}
Strongly cancellative sets $\C_{B_n}$ are defined in section 2. In this section, we show that the maximal size of $\C_{B_n}$ on $B_n$ is $2^{\lfloor \frac{n}{2} \rfloor}$. 

\begin{theorem}\label{tconcans}
There exists a strongly cancellative set $\C_{B_n}$ of size $2^{\lfloor \frac{n}{2} \rfloor}$ on $B_n$.
\end{theorem}

\begin{proof}
We construct $\C_{B_n}$ as follows. First, let us divide the set $\left\{1,\ldots ,2\lfloor\frac{n}{2}\rfloor \right\}$ into $\lfloor\frac{n}{2}\rfloor$ blocks $S_i = \{2i -1,2i \}$, for $1 \leq i \leq \lfloor\frac{n}{2}\rfloor$. We define $\C_{B_n}$ to be the family of all the subsets $s = \left\{ s_1,\ldots,s_{\lfloor\frac{n}{2}\rfloor} \right\}$ such that $s_i \in S_i,\text{ for } 1 \leq i \leq \lfloor\frac{n}{2}\rfloor$. It is easy to see that $|\C_{B_n}| =2^{\lfloor\frac{n}{2}\rfloor}$. Now, we show that $\C_{B_n}$ satisfies the conditions defining strongly cancellative set. 

Consider different elements $b = \left\{b_1,\ldots,b_{\lfloor\frac{n}{2}\rfloor}\right\}$ and $c = \left\{c_1,\ldots,c_{\lfloor\frac{n}{2}\rfloor}\right\}$ in $\C_{B_n}$, so that there exists some $1 \leq k \leq \lfloor\frac{n}{2}\rfloor$ such that $b_k \neq c_k$. Without lost of generality, assume that $b_k = 2k-1$ and $c_k = 2k$. Consequently, for any element $a = \left\{a_1,\ldots,a_{\lfloor\frac{n}{2}\rfloor}\right\}$, we have the following properties:

\begin{enumerate}
	\item $b_k \notin a \cap c$ and $c_k \notin a\cap b$,
	\item $b_k \in a \cup b$ and $c_k \in a \cup c$,
 	\item $a_k = b_k$ or $a_k = c_k$,
	\item $b_k \in a \cap b$ or $c_k \in a \cap c$,
	\item $c_k \notin a \cup b$ or $b_k \notin a \cup c$.
\end{enumerate}

Clearly, property (3) implies (4) and (5). Moreover, (1) and (4) imply that $a \cap b \neq a \cap c$, and similarly, (2) and (5) imply that $a \cup b \neq a \cup c$. Therefore, $\C_{B_n}$ is a strongly cancellative set.
\end{proof}

Now, we show that $\left|\C_{B_n}\right| \leq 2^{\left\lfloor \frac{n}{2} \right\rfloor}$.

\begin{theorem}\label{tiecans}
For any strongly cancellative $\C_{B_n}$ on $B_n$, we have $|\C_{B_n}| \leq 2^{\left\lfloor \frac{n}{2} \right\rfloor}$.
\end{theorem}

\begin{proof}
Fix an element $v' \in \C_{B_n}$. We consider the following sets:
\begin{align*}
\C_{1}(v') &= \{ v \cap v': v \neq v',v\in \C_{B_n} \}, \\
\C_{2}(v') &= \{ v \cup v': v \neq v',v\in \C_{B_n} \}.
\end{align*}
By Equation \eqref{e6}, we have $\left|\C_{1}(v')\right| = \left|\C_{2}(v')\right| = |\C_{B_n}| -1$. This implies that
\begin{equation}|\C_{B_n}| \leq 1 + \min\left(\left|\{ v : v \subseteq v' \}|,|\{ v : v \supseteq v' \}\right|\right).\end{equation}
Moreover, it is not hard to show that 
\begin{equation}\label{e7}
\min\left(\left|\{ v : v \subseteq v' \}\right|,\left|\{ v : v \supseteq v' \}\right|\right) \leq \left|\left\{ v : v \subseteq v^{*}, \text{rank}(v^{*}) = \left\lfloor\frac{n}{2} \right\rfloor\right\}\right|.
\end{equation}
We consider the following two cases:
\begin{enumerate}
\item $2 \mid n$. Then we have $\text{rank}(v') = \left\lfloor \frac{n}{2} \right\rfloor$ if the equality holds in (4.2). Suppose that the equalities in (4.1) and (4.2) hold for every $v' \in \C_{B_n}$. Consequently, $\text{rank}(v') = \left\lfloor \frac{n}{2} \right\rfloor$, for every $v' \in \C_{B_n}$, which implies that any two elements in the set are incomparable. One can easily see that, $\C_{1}(v') \neq \left|\{ v : v \subseteq v' \}\right|$ and $\C_{2}(v') \neq \left|\{ v : v \supseteq v' \}\right|$. Therefore, the equalities in (4.1) and (4.2) can not hold at the same time.
\item $2 \nmid n$. Then $\text{rank}(v') = \left\lfloor \frac{n}{2} \right\rfloor$ or $ \left\lfloor \frac{n+1}{2} \right\rfloor$ if the equality holds in (4.2). Suppose that the equalities in  (4.1) and (4.2) holds for every $v' \in \C_{B_n}$. Pick some element $w \in \C_{B_n}$. If $\text{rank}(w) = \left\lfloor \frac{n}{2} \right\rfloor$, then by (4.1) there exists another two elements $w'$ and $w''$ in the set such that $w \cap w' = w$ and $w\cap w'' = \emptyset$. This implies that $\text{rank}(w') = \left\lfloor \frac{n+1}{2} \right\rfloor$ and $w' \backslash w = \{x \}$, where $1\leq x\leq n$. 

By Equation \eqref{e6}, we have $\emptyset = w\cap w'' \neq w'\cap w''$, and thus $x \in w''$. This means that $w\cup w'' = w'\cup w''$ which is not possible. As a result, the equalities in (4.1) and (4.2) cannot hold at the same time. Similarly, one can prove the same statement when $\text{rank}(w) = \left\lfloor \frac{n+1}{2} \right\rfloor$.
\end{enumerate}

Finally, from (1) and (2), we have 
$$|\C_{B_n}| \leq \left|\left\{ v : v \subseteq v^{*}, \text{rank}(v^{*}) = \left\lfloor\frac{n}{2} \right\rfloor\right\}\right| = 2^{\left\lfloor \frac{n}{2} \right\rfloor}.$$
\end{proof} 
 
\section{\textbf{Strongly Cancellative Sets on $D_{l_1,\ldots,l_k}$ and $D_{l}^{k}$}}
For the definition of $D_{l_1,\ldots,l_k}$, see section 2. In particular, we say that $D_{l}^{k}$ is the lattice of $k$ chains of length $l-1$. Clearly, $D_{2}^{n}$ is  Boolean lattice $B_n$, and it is easy to show that for any two points $v = (v_1,\ldots,v_k)$ and $v' = (v'_1,\ldots,v'_k)$ in $D_{l_1,\ldots,l_k}$,
\begin{align*}
 (v_1,\ldots,v_k) \wedge  (v'_1,\ldots,v'_k) &= \big{(}\min{(v_1,v'_1)},\ldots,\min{(v_k,v'_k)}\big{)}, \\ 
 (v_1,\ldots,v_k) \vee  (v'_1,\ldots,v'_k) &= \big{(}\max{(v_1,v'_1)},\ldots,\max{(v_k,v'_k)}\big{)}  . 
\end{align*}

In the following proposition, we give a tight bound for the size of strongly cancellative sets on $D_{l_1,l_2}$.

\begin{proposition}\label{ptbcans}
Let $\C_{D_{l_1,l_2}}$ be a strongly cancallative set on $D_{l_1,l_2}$. Then $$\left|\C_{D_{l_1,l_2}}\right| \leq \min(l_1,l_2).$$
\end{proposition}

\begin{proof}
Without lost of generality, we assume that $l_1 \leq l_2$. Every point $v$ in $D_{l_1,l_2}$ is a vector $(v_1,v_2)$, where $0 \leq v_1 \leq l_1 -1$ and  $0 \leq v_2 \leq l_2 - 1$. We proceed by contradiction.

Suppose that $\left|\C_{D_{l_1,l_2}}\right| > l_1$. Then there exists two points $v = (v_1,v_2)$ and $w = (w_1,w_2)$ such that $v_1 = w_1$ and $v_2 < w_2$. For any point $v^* = (v^*_1,v^*_2) \notin \{v, w\}$, all the following four possible cases lead to contradiction:

\begin{enumerate}
	\item $v^*_2 \leq v_2$ implies that $v^* \wedge v = v^* \wedge w$.
	\item $v^*_2 > w_2$ implies that $v^* \vee v = v^* \vee w$.
	\item $v_2 \leq v^*_2 \leq w_2$ and $v^*_1 \leq v_1$ imply that $v^* \vee w = v \vee w$.
	\item $v_2 \leq v^*_2 \leq w_2$ and $v^*_1 \geq v_1$ imply that $v^* \wedge v = v \wedge w$.
\end{enumerate}

Therefore, we must have $\left|\C_{D_{l_1,l_2}}\right| \leq l_1 = \min(l_1,l_2)$, as desired.
\end{proof}
The bound $\min(l_1,l_2)$ is tight for $\left|\C_{D_{l_1,l_2}}\right|$. In particular, it is not hard to show that the following set is a strongly cancellative set of size $\min(l_1,l_2)$:
$$\C_{D_{l_1,l_2}} = \{ (x,y) \mid x+y = \min(l_1,l_2)-1 \}.$$

In the following, we study the size of the strongly cancellative sets on $D_l^k$.

\begin{proposition}\label{pconcans}
Suppose that $\C_{k_1}$ is a strongly cancellative set on $D_l^{k_1}$ for some small $k_1$, and any two elements in $\C_{k_1}$ are incomparable. Then, for any $k$ with $\left\lfloor\frac{k}{k_1}\right\rfloor= s$, there is a strongly cancellative set $\C_k$ of size $|\C_{k_1}|^s$ on $D_l^{k}$.
\end{proposition}

\begin{proof}
Every point in $D_l^{k}$ is a $k$-dimensional vector $(v_1,\ldots,v_k)$, where $0\leq v_i\leq l-1$ for $1\leq i \leq k$. For every vector $v = (v_1,\ldots,v_k)$, we define subvectors induced by $v$ as $B_j(v) = (v_{(j-1)k_1+1},\ldots,v_{jk_1})$, for $1\leq j\leq s$, and $B_{s+1}(v) = (v_{k_1s+1},\ldots,v_{k})$. Let $\C_k$ to be the set of all $k$-dimensional vectors $v$ such that $B_j(v) \in \C_{k_1}$ for all $1\leq j \leq s$, and $B_{s+1}(v)$ is the zero vector. Clearly, we have $|\C_k| = \left|\C_{k_1}\right|^s$.

Suppose there are three different elements $v,v',v'' \in \C_k$ such that $v \vee v' = v\vee v''$. Since $v'$ and $v''$ are different, we have $B_{j^*}(v') \neq B_{j^*}(v'')$ for some $1\leq j^* \leq s$. On the other hand, we know $B_{j^*}(v) \vee B_{j^*}(v') = B_{j^*}(v) \vee B_{j^*}(v'')$ which implies that one of $B_{j^*}(v')$ or $B_{j^*}(v'')$ is equal to $B_{j^*}(v)$. Therefore, $v'_i \preceq v''_i$ or $v''_i \preceq v'_i$, and this contradicts the assumption that any two different elements in $\C_{k_1}$ are incomparable. Similarly, it is easy to see that $v \wedge v' \neq v\wedge v''$. As a result, $\C_k$ is a strongly cancellative set of size $\left|\C_{k_1}\right|^s$.
\end{proof}

We can use this result to give a construction of a strongly cancellative set on $D_{l}^{k}$.

\begin{corollary}\label{tcandnc}
There exists a strongly cancellative set $\C_{D_{l}^{k}}$ on $D_{l}^{k}$, such that $\left|\C_{D_{l}^{k}}\right| = l^{\lfloor \frac{k}{2} \rfloor}$.
\end{corollary}

\begin{proof}
We have seen that $\C_{D_{l}^{2}} = \{ (x,y) \mid x+y = l-1 \}$ is a strongly cancellative set of size size $l$ on $D_{l}^{2}$ such that any two elements in the set are incomparable. By Proposition \ref{pconcans}, there exists a strongly cancellative set $\C_{D_{l}^{k}}$ of size $l^{\lfloor \frac{k}{2} \rfloor}$.
\end{proof}

Finally, we show an upper bound for the size of strongly cancellative sets on $D_{l}^{k}$.

\begin{theorem}\label{tcandnie}
Let $\C_{D_{l}^{k}}$ be a strongly cancellative set on $D_{l}^{k}$, then 
$$\left|\C_{D_{l}^{k}}\right| \leq (2 l)^{\frac{k}{2}} + \frac{k(l-1)}{2} +1.$$
\end{theorem}

\begin{proof}
Any element $v$ on the lattice $D_{l}^{k}$ is a $k$-dimensional vector $v = (v_1,\ldots,v_k)$ such that $0\leq v_i \leq l-1$ for all $1 \leq i \leq k$. We first define $\C_m(t)$ and $\Pb_m(t)$.

\begin{enumerate}
	\item We define $\C_m(t)$ to be set of vectors whose $m$-th compoenent is $t$, for any $1 \leq t \leq k$. That is, $\C_m(t) = \{ v \mid v \in \C_{D_{l}^{k}}, \text{ } v_m = t  \}.$
	\item Let $v$ be a random element uniformly chosen in the set $\C_{D_{l}^{k}}$. We denote the probability that the $m$-th component $v_m$ of $v$ is $t$ by $\Pb_m(t)$. So, 
$$\Pb_m(t) = \frac{|\C_m(t)|}{\left|\C_{D_{l}^{k}}\right|}.$$
\end{enumerate}

Fix an arbitrary element $v \in \C_{D_{l}^{k}}$. We define the random variable $\xi_v = v \wedge v^*$, where $v^*$ is the random element uniformly chosen in  $\C_{D_{l}^{k}}\backslash \{ v \}$. Suppose that there exists two element $v_1$ and $v_2$ in $\C_{D_{l}^{k}}$ so that obtain the same value in $\xi_v$. That is, $v \wedge v_1 = v \wedge v_2$ which is not possible in strongly cancellative sets. Consequently, every value in $\xi_v$ appears exactly once. Since there are totally $\left|\C_{D_{l}^{k}}\right|-1$ different values for $\xi_v$, the entropy function of $\xi_v$ is
\begin{equation}\label{e4}
 \mathcal{H}(\xi_v) = \log{\left( \left|\C_{D_{l}^{k}}\right| -1  \right)}.
\end{equation}
For convenience, we set $N = \left|\C_{D_{l}^{k}}\right| -1$.

On the other hand, every value in $\xi_v$ is a $k$-dimensional vector $(\xi_v(1),\ldots,\xi_v(k))$ such that $\xi_v(m) = \min(v_m,v^{*}_m)$ for any $1\leq m\leq k$ and randomly chosen element $v^{*}$. Consequently, for any $1\leq m\leq k$, $\xi_v(m)$ takes all its values in $\{0,1,\ldots,v_m  \}$. We denote the probability that $\xi_v(m) = t'$ by $\Pb_{\xi_v(m)}(t')$. Moreover, if $0 \leq t' \leq v_m -1$, we should have $t'=\min(v_m,v^{*}_m) < v_m$ and thus, $v^{*}_m = t'$. If $t' = v_m$, we must have  $\min(v_m,v^{*}_m) = t' = v_m$ which implies that $v_m\leq v^{*}_m$.

Therefore, we obtain the following properties for $P_{\xi_v(m)}(t')$:
\begin{equation}
\Pb_{\xi_v(m)}(t') = 
\begin{cases} 
 \hfill \frac{|\C_m(t')|}{N},  & 0 \leq t' \leq v_m -1. \\
 \hfill \frac{\sum_{t'_1 = v_m}^{l-1} |\C_m(t')| -1}{N}, & t' = v_m. \\
 \hfill 0, & v_m+1 \leq t' \leq l-1.
\end{cases}
\end{equation}
The entropy function of $\xi_v(m)$ can be computed as follows:
\begin{align*}
\mathcal{H}\left(\xi_v(m)\right)
 & = \mathcal{H}\left(\Pb_{\xi_v(m)}(0),\ldots,\Pb_{\xi_v(m)}(v_m-1),\Pb_{\xi_v(m)}(v_m)\right) \\
 & = \sum_{t' = 0}^{v_m}\Pb_{\xi_v(m)}(t')\log{\frac{1}{\Pb_{\xi_v(m)}(t')}}.
\end{align*}
Furthermore, by Eq.\eqref{e4} and Theorem \eqref{eentropyin}, we have 
\begin{equation}
\log{N} \leq \sum_{m=1}^{k}\mathcal{H}\left(\xi_v(m)\right)
  = \sum_{m=1}^{k}\sum_{t' = 0}^{v_m}\Pb_{\xi_v(m)}(t')\log{\frac{1}{\Pb_{\xi_v(m)}(t')}}.
\end{equation}

Since the above equation holds for every element $v$ in the set $\C_{D_{l}^{k}}$. If we take the average over all the elements in the set $\C_{D_{l}^{k}}$, we obtain
\begin{equation}\label{e5}
\log{N} \leq \frac{\sum_{v \in \C_{D_{l}^{k}}}\sum_{m=1}^{k}\mathcal{H}\left(\xi_v(m)\right)}{N+1} =  \frac{\sum_{m=1}^{k}\sum_{v \in \C_{D_{l}^{k}}}\mathcal{H}\left(\xi_v(m)\right)}{N+1}.
\end{equation}
Moreover, from (2), we know that the probability that $v_m = t$ for some  $0 \leq t \leq l-1$ is $\Pb_m(t) = \frac{|\C_m(t)|}{N+1}$, and therefore, \eqref{e5} can be rewritten as follows:
\begin{equation}\label{ecandis1}
\log{N} \leq \sum_{m=1}^{k}\sum_{t = 0}^{l-1}\Pb_m(t)\left(\sum_{t' = 0}^{t}P_{\xi_v(m)}(t')\log{\frac{1}{\Pb_{\xi_v(m)}(t')}}\right).
\end{equation}

Now, we consider the random variable $\xi_v^{'} = v \vee v^*$, where $v^*$ is also independently chosen under the uniform distribution on $\C_{D_{l}^{k}}\backslash \{ v \}$. Thus, we have the following:
\begin{equation}\label{casescans}
\Pb_{\xi'_v(m)}(t') = 
\begin{cases} 
 \hfill 0,  & 0 \leq t' \leq v_m -1. \\
 \hfill \frac{\sum_{t'_1 = 0}^{v_m} |\C_m(t')| -1}{N}, & t' = v_m. \\
 \hfill \frac{|\C_m(t')|}{N}, & v_m+1 \leq t' \leq l-1.
\end{cases}
\end{equation}
By similar arguments, Eq.\eqref{casescans} implies that:
\begin{equation}\label{ecandis2}
\log{N} \leq \sum_{m=1}^{k}\sum_{t = 0}^{l-1}\Pb_m(t)\left(\sum_{t' = t}^{l-1}\Pb_{\xi_v(m)}(t')\log{\frac{1}{\Pb_{\xi_v(m)}(t')}}\right).
\end{equation}
For convenience, we set $\Pb_{m}(t')= \frac{|\C_m(t')|}{N}$, $q _m(t) = \frac{\sum_{t'_1 = t}^{l-1} |\C_m(t')| -1}{N}$, and $q '_m(t) = \frac{\sum_{t'_1 = 0}^{t} |\C_m(t')| -1}{N}$. Consider the following inequality,
\begin{align}
& \sum_{t' = 0}^{t}P_{\xi_v(m)}(t')\log{\frac{1}{\Pb_{\xi_v(m)}(t')}} + \sum_{t' = t}^{l-1}\Pb_{\xi_v(m)}(t')\log{\frac{1}{\Pb_{\xi_v(m)}(t')}}\\
\leq & \left(\sum_{t'=0}^{l-1}\Pb '_m(t')\log{\frac{1}{\Pb '_m(t')}}\right)+ q _m(t)\log{\frac{1}{q_m(t)}} + q '_m(t)\log{\frac{1}{q'_m(t)}} \\
\leq & \left(\frac{N+1}{N} \right) \log{\frac{l   N}{N+1}} +\left(q_m(t)+q'_m(t)\right) \cdotp \log{\left(\frac{2}{q_m(t)+q'_m(t)}\right)}.
\end{align}
Note that (5.9) holds because $p\log{\frac{1}{p}} >0$, when $0<p<1$, and (5.10) holds by the inequality in \eqref{ecommonb}.

\vspace{5pt}
Finally, by adding \eqref{ecandis1} and \eqref{ecandis2}, the above result implies that
\begin{align*}
2\log{N}  &\leq \sum_{m=1}^{k}\sum_{t = 0}^{l-1} \Pb_m(t) \left[ \left(q_m(t)+q'_m(t)\right) \cdotp \log{\left(\frac{2}{q_m(t)+q'_m(t)}\right)}+ \left(  1+\frac{1}{N} \right)  \log{l} \right]  \\
&= k \left(  1+\frac{1}{N} \right)  \log{l} +\sum_{m=1}^{k}\sum_{t = 0}^{l-1} \Pb_m(t)\cdotp \left(q_m(t)+q'_m(t)\right) \cdotp \log{\left(\frac{2}{q_m(t)+q'_m(t)}\right)}   \\
&\leq k + k \left(  1+\frac{1}{N} \right)  \log{l}.
\end{align*}
The last inequality is due to the fact that function $x\log{\frac{2}{x}}$ is decreasing with $x \geq 1$ and that $q_m(t) + q'_m(t) = 1+ \frac{|\C_m(t)|-1}{N} \geq 1$ when $\Pb_m(t) = \frac{|\C_m(t)|}{N+1}\neq 0$. 
\vspace{5pt}

Therefore, we have
\begin{equation}\label{e2}
N \leq 2^{\frac{k}{2}} l^{\frac{k}{2}\left(1+\frac{1}{N}\right)}.
\end{equation}
Consider the function $f(N) = N -  2^{\frac{k}{2}} l^{\frac{k}{2}\left(1+\frac{1}{N}\right)}$. The inequality \eqref{e2} implies that $f(N) \leq 0$ and is increasing with $N$. If we set $N_1 =   (2 l)^{\frac{k}{2}} + \frac{k(l-1)}{2}$, then it is easy to see that
\begin{align}\label{e3}
f(N_1) &= \frac{k(l-1)}{2} +  (2 l)^{\frac{k}{2}}\cdotp \left(1-(1+l-1)^{\frac{k}{2N_1}}\right)\\
 &\geq \frac{k(l-1)}{2} + (2 l)^{\frac{k}{2}}\cdotp \left(1-\left(1+ \frac{(l-1)k}{2N_1}\right)\right) \\
& =  \frac{k(l-1)}{2} - \frac{(2 l)^{\frac{k}{2}}}{N_1} \cdotp \frac{k(l-1)}{2} \geq 0,
\end{align}
where (5.12) implies (5.13) because $(1+a)^b \leq 1+ab$ when $b \leq 1$ and $a\geq 0$.
\vspace{5pt}

As a result, since $f(N) \leq 0 \leq f(N_1)$,
$$\left|\C_{D_{l}^{k}}\right| -1 = N \leq N_1 = (2 l)^{\frac{k}{2}} + \frac{k(l-1)}{2}.$$
\end{proof}

\end{document}